\documentclass[12pt,leqno]{article}

\usepackage{latexsym,amsmath,amssymb,amsthm,amsfonts}
\usepackage[all]{xy}

\renewcommand{\epsilon}{\varepsilon}

\newcommand{\R}{\mathbb{R}}

\newcommand{\p}{\mathcal{P}}

\newtheorem{theorem}{Theorem}
\newtheorem{lemma}{Lemma}

\begin{document}

\author{Andriy V.\ Bondarenko
\footnote{Part of this work was done while the first author was on a
visit at Max Planck Institute for Mathematics, Bonn, Germany in April-May, 2008} and Maryna S. Viazovska}
\title{Spherical designs via Brouwer fixed point theorem}
\date{}
\maketitle
\noindent 1429 Stevenson Center\\
Vanderbilt University\\
Nashville, TN 37240\\
Tel. 615-343-6136\\
Fax 615-343-0215\\
\vspace{1cm}
Email: andriy.v.bondarenko@Vanderbilt.Edu\\
Max Planck Institute for Mathematics,\\
Vivatsgasse 7, 53111 Bonn, Germany\\ Tel. +49-228-402-265\\
Fax +49-228-402-275\\
Email: viazovsk@mpim-bonn.mpg.de
\newpage
\begin{abstract}
For each $N\ge c_dn^\frac{2d(d+1)}{d+2}$ we prove the existence of
a spherical $n$-design on $S^d$ consisting of $N$ points, where
$c_d$ is a constant depending only on $d$.
\end{abstract}
{\sl Keywords:} Spherical designs, Brouwer fixed point theorem,
Marcinkiewich-Zygmund inequality, area-regular partitions.\\

\newpage
\section {Introduction}

Let $S^d$ be the unit sphere in $\R^{{d+1}}$ with normalized
Lebesgue measure $d\mu_d$ $\left(\int_{S^d}d\mu_d(x)=1\right)$.
The following concept of a spherical design was introduced by
Delsarte, Goethals and Seidel ~\cite{DGS}:\\ A set of points
$x_1,\ldots,x_N\in S^d$ is called a {\sl spherical $n$-design} if
$$ \int_{S^d}P(x)d\mu_d(x)=\frac 1N\sum_{i=1}^N P(x_i) $$ for all
algebraic polynomials in $d+1$ variables and of total degree at
most $n$. For each $n\in {\mathbb N}$ denote by $N(d,n)$ the
minimal number of points in a spherical $n$-design.  The following
lower bounds
\begin{equation}
\label{hh}
N(d,n)\ge {{d+k}\choose{d}}+{{d+k-1}\choose{d}},\quad
n=2k,
\end{equation}
$$ N(d,n)\ge 2\,{{d+k}\choose{d}}, \quad n=2k+1, $$ are also
proved in \cite{DGS}.

Spherical $n$-designs attaining these bounds are called tight.
Exactly eight tight spherical designs are known for $d\ge 2$ and
$n\ge 4$. All such configurations of points are highly symmetrical
and possess  other extreme properties. For example, the shortest
vectors in the $E_8$ lattice form a tight 7-design in $S^7$, and a
tight 11-design in $S^{23}$ is obtained  from the Leech lattice in
the same way ~\cite{CS}. In general, lattices are a good source
for spherical designs with small $(d,n)$~\cite{PPV}.

 On the
other hand construction of spherical $n$-design with minimal
cardinality for fixed $d$ and $n\to\infty$ becomes a difficult
analytic problem even for $d=2$. There is a strong relation
between this problem and the problem of findind $N$ points on a
sphere $S^2$ that minimize the energy functional $$
E(\vec{x}_1,\ldots,\vec{x}_N)=\sum_{1\le i<j\le N}\frac
1{\|\vec{x}_i-\vec{x}_j\|},$$ see Saff, Kuijlaars~\cite{SK}.

 Let us begin by giving a short history of asymptotic upper bounds on $N(d,n)$ for fixed
$d$ and $n\to\infty$. First, Seymour and Zaslavsky \cite{SZ} have
proved that spherical design exists for all $d$, $n\in {\mathbb
N}$. Then, Wagner \cite{Wag} and Bajnok \cite{B} independently
proved that $N(d,n)\le c_dn^{Cd^4}$ and $N(d,n)\le c_dn^{Cd^3}$
respectively. Korevaar and Meyers have \cite{KM} improved this
inequalities by showing that $N(d,n)\le c_dn^{(d^2+d)/2}$. They
have also conjectured that $N(d,n)\le c_dn^d$. Note that
\eqref{hh} implies $N(d,n)\ge C_dn^d$. In what follows we denote
by $b_d$, $c_d$, $c_{1d}$, etc., sufficiently large constants
depending only on $d$. In~\cite{BV} we  proved the following\\
{\bf Theorem BV.} Let $a_d$ be the sequence defined by $$
a_1=1,\quad a_2=3, \quad a_{2d-1}=2a_{d-1}+d, \quad
a_{2d}=a_{d-1}+a_{d}+d+1,\quad d\ge 2. $$ Then for all $d$, $n\in
{\mathbb N}$, $$ N(d,n)\le c_dn^{a_d}. $$
\\ {\bf Corollary BV.} For each $d\ge 3$ and $n\in {\mathbb N}$ we have
$$ N(d,n)\le c_dn^{a_d}. $$ $$a_3\le 4, \quad a_4\le 7, \quad
a_5\le 9, \quad a_6\le 11, \quad a_7\le 12, \quad a_8\le 16, \quad
a_9\le 19, \quad a_{10}\le 22,$$ and $$ a_d< \frac d2\log_22d,
\quad d> 10. $$ In this paper we suggest a new nonconstructive
approach for obtaining new upper bounds for $N(d,n)$. We will make
extensive use of the Brouwer fixed point theorem (the source of
nonconstructive nature of our method), the Marcinkiewich-Zygmund
inequality on the sphere~\cite{MNW} and the notion of area-regular
partitions~\cite{SK2}. The main result of this paper is \\
\begin{theorem}\label{main} For each $N\ge c_dn^\frac{2d(d+1)}{d+2}$ there exists a spherical $n$-design on
$S^d$ consisting of $N$ points.
\end{theorem}
This result improves our previous estimate on $N(d,n)$ for all
$d>3$, $d\neq 7$, and in particular allows us to remove the
"nasty" logarithm in the power in Corollary BV, so that the
function in the power has a linear behavior, which confirms the
conjecture of Korevaar and Meyers. Finally, Theorem~1 guaranties
the existence of spherical $n$-design for each $N$ greater then
our new existence bound.

\section{Preliminaries}
Let $\Delta$ be the Laplace operator in $\mathbb{R}^{d+1}$ $$
\Delta =\sum_{j=1}^{d+1}\frac{\partial^2}{\partial x_j^2}. $$ We
say that a polynomial $P$ in $\mathbb{R}^{d+1}$ is harmonic if
$\Delta P=0$. For integer $k\ge 1$, the restriction to $S^d$ of a
homogeneous harmonic polynomial of degree $k$ is called a
spherical harmonic of degree $k$. The vector space of all
spherical harmonics of degree $k$ will be denoted by
$\mathcal{H}_k$ (see~\cite{MNW} for  details). The dimension of
$\mathcal{H}_k$ is given by $$ \dim\,\mathcal{H}_k=
\frac{2k+d-1}{k+d-1}\binom{d+k-1}{k}. $$  The vector spaces
$\mathcal{H}_k$ are invariant under the action of the orthogonal
group $O(d+1)$ on $S^d$ and are orthogonal to each other with
respect to the scalar product $$\langle P,
Q\rangle:=\int_{S^d}P(x)Q(x)d\mu_d(x) .$$ Another remarkable
property of harmonic polynomials is that the spaces
$\mathcal{H}_k$ are eigenspaces of  the spherical Laplacian
(Laplace-Beltrami operator~\cite{Dz})
\begin{equation}
\label{Be}
\widetilde{\Delta}f(x):=\Delta f(\frac
{x}{\|x\|}).
\end{equation}
Thus, for a polynomial $P\in\mathcal{H}_k$ we have
\begin{equation}\label{eigv}
\widetilde{\Delta}P=-k(k+d-1)P.
\end{equation}
Here and below we use the notations $\|\cdot\|$ and $(\cdot,
\cdot)$ for the Euclidean norm and usual scalar product in
$\mathbb{R}^{d+1}$, respectively. For a twice differentiable
function\\ $f: \R^{d+1}\to \R$ and a point
$x_0\in\mathbb{R}^{d+1}$ denote by $$ \frac{\partial f}{\partial
x}(x_0):=\left(\frac{\partial f}{\partial x_1}(x_0),\ldots
\frac{\partial f}{\partial x_{d+1}}(x_0)\right) $$ and $$
\frac{\partial^2 f}{\partial x^2}(x_0):=\left(\frac{\partial^2
f}{\partial x_i\partial x_j}(x_0)\right)_{i,j=1}^{d+1} $$ the
gradient and the matrix of second derivatives of $f$ (Hessian
matrix) at the point $x_0$ respectively. Analogously to~\eqref{Be}
we will also define for a polynomial $Q\in\p_n$ the spherical
gradient $$\nabla Q(x):=\frac{\partial}{\partial x}Q(\frac
{x}{\|x\|})$$ and the Hessian matrix on the sphere
\begin{equation}\label{Hes}
\nabla^2 Q(x):=\frac{\partial^2 }{\partial x^2}Q(\frac
{x}{\|x\|}).
\end{equation} We will also write $$\nabla^2Q\cdot x\cdot
y:=(\nabla^2Q\cdot x, y)\qquad\mbox{for}\;x,y\in \mathbb{R}^{d+1}.
$$
  One consequence of Stokes's theorem is the
first Green's identity \cite{W}
\begin{equation}\label{grad}
\int_{S^d}P(x)\widetilde{\Delta}Q(x)d\mu_d(x)=-\int_{S^d}( \nabla
P(x),\nabla Q(x)) d\mu_d(x).
\end{equation}

 Let $\mathcal{P}_n$ be the vector space of
polynomials $P$ of degree $\leq n$ on $S^{d}$ such that $$
\int_{S^d}P(x)d\mu_d(x)=0. $$ Each polynomial in
$\mathbb{R}^{d+1}$ can be written as a finite sum of terms, each
of which is a product of a harmonic and a radial polynomial (i.e.
a polynomial which depends only on $\|x\|$). Therefore the vector
space $\p_n$ decomposes into the direct sum  $\mathcal{H}_k$
$$\p_n=\bigoplus_{k=1}^n \mathcal{H}_k.$$ For each vector of
positive weights $w=(w_1,\ldots,w_n)$ we can define a scalar
product $\langle\cdot,\cdot \rangle_w$ on $\p_n$ invariant with
respect to the action of $O(d+1)$ on $S^d$ by $$\langle P,Q
\rangle_w:=\sum_{k=1}^n w_k\langle P_k, Q_k\rangle,$$ where $P_k$,
$Q_k\in\mathcal{H}_k$, $P=P_1+\ldots+P_n$ and $Q=Q_1+\ldots+Q_n$.
For each $Q\in\p_n$ denote by $$ \|Q\|_w=\sqrt{\langle
Q,Q\rangle_w} $$ the norm corresponding to this scalar product. We
will also define the operator
$$\Delta_wP:=\sum_{k=1}^n\frac{k(k+d-1)}{w_k}P_k,\;P\in\p_n.$$
Then from (\ref{eigv}) and (\ref{grad}) we get
\begin{equation}
\label{lap}
\langle \Delta_w
P,Q\rangle_w=\int_{S^d}\langle \nabla P(x),\nabla Q(x)\rangle
d\mu_d(x).
\end{equation}

  Now, for each point $x\in S^d$ there exists a unique
polynomial $G_x\in \p_n$ (depending on $w$) such that $$\langle
G_x,Q\rangle_w=Q(x) \;\;\mbox{for all}\;\;Q\in\p_n.$$ Then, the
set of points $x_1,\ldots,x_N\in S^d$ form a spherical design if
and only if $$G_{x_1}+\ldots+G_{x_N}=0 .$$  To construct the
polynomials $G_x$ explicitly we will use the  Gegenbauer
polynomials $G^{\alpha}_k$~\cite{AS}. For a fixed $\alpha$, the
$G^{\alpha}_k$ are orthogonal on $[-1,1]$ with respect to the
weight function $\omega (t)=(1-t^2)^{{\alpha}-\frac12}$, that is
$$
\int_{-1}^1G^{\alpha}_m(t)G^{\alpha}_n(t)(1-t^2)^{{\alpha}-\frac12}dt=\delta_{mn}
\frac{\pi2^{1-2\alpha}\Gamma(n+2\alpha)}{n!(\alpha+n)\Gamma^2(\alpha)}.
$$ Set $\alpha:=\frac{d-1}2$, and let $$G_x(y):=g_w((x,y)),$$
where  $$
g_w(t):=\sum_{k=1}^n\frac{\dim\,\mathcal{H}_k}{w_kG^{\alpha}_k(1)}G^{\alpha}_k(t).
$$ In order to show that $\langle P_x,Q\rangle_w=G_x(Q)=Q(x)$ for
each $Q\in\p_n$ we will use the following identity for Gegenbauer
polynomials~\cite{MNW}
\begin{equation}
\label{g1}
G^{\alpha}_k((x,y))=\frac{G^{\alpha}_k(1)}{\dim\,\mathcal{H}_k}\sum_{j=1}^{\dim\,\mathcal{H}_k}
Y_{jk}(x)Y_{jk}(y),
\end{equation}
where $x,y\in S^d$ and $Y_{jk}$ are some orthonormal basis in the
space $(\mathcal{H}_k, \mu_d )$. In particular, for a fixed $x\in
S^d$, $G^{\alpha}_k((x,y))\in\mathcal{H}_k$. Therefore, for a
polynomial  $Q\in\p_n$ we have $$ \langle
G_x,Q\rangle_w=\sum_{k=1}^n w_i\langle G_k, Q_k\rangle=
\sum_{k=1}^n\int_{S^d}G^{\alpha}_k((x,y))Q_k(y)d\mu_d(y)= $$ $$
=\sum_{k=1}^n\sum_{j=1}^{\dim\,\mathcal{H}_k}Y_{jk}(x)\int_{S^d}Q_k(y)Y_{jk}(y)d\mu_d(y)
=\sum_{k=1}^nQ_k(x)=Q(x). $$

Fix the weight vector $w=(w_1,\ldots,w_n)$ such that
$w_k=k(k+d-1)$. Further we will use the following additional
equalities for Gegenbauer polynomials~\cite{AS}: $$
G^{\alpha}_n(1)=\binom{2\alpha+n-1}{n}, $$ and
\begin{equation}
\label{g2}
\frac d{dt}G^{\alpha}_n(t)=2\alpha G^{\alpha+1}_{n-1}(t), \qquad
\frac {d^2}{dt^2}G^{\alpha}_n(t)=4\alpha(\alpha+1)G^{\alpha+2}_{n-2}(t).
\end{equation}
Applying Cauchy's  inequality to ~\eqref{g1} we get, for all
$k\in\mathbb{N}$ and $x,y\in S^d$, $$ |G^{\alpha}_k((x,y))|^2\le
G^{\alpha}_k((x,x))G^{\alpha}_k((y,y)), $$ and hence $$
\max_{x\in[-1,1]}|g_w(x)|=g_w(1).$$ Similarly, by ~\eqref{g2} we
obtain
\begin{equation}
\label{g3} \max_{x\in[-1,1]}|g'_w(x)|=g'_w(1).
\end{equation}
Finally, let us estimate $g_w'(1)$ and $g_w''(1)$. We have
\begin{equation}
\label{g4}
g_w'(1)=\sum_{k=1}^n\frac{\dim\,\mathcal{H}_k}{w_kG^{\alpha}_k(1)}{G_k^{\alpha}}^{\prime}(1)=
\sum_{k=1}^n\frac{(2k+d-1)(k+d-2)!}{k!d!}\le c_{1d}n^d.
\end{equation}
Hence, by~\eqref{g3} and Markov inequality we get
\begin{equation}
\label{g5} g''_w(1)<n^2\max_{x\in[-1,1]}|g'_w(x)|=n^2g'_w(1)\le
c_{1d}n^{d+2}.
\end{equation}
\section{Proof of Theorem \ref{main}}

Fix $n\in\mathbb{N}$. As  mentioned in section 2, points
$x_1,\ldots,x_N$ form a spherical $n$-design if and only if
$G_{x_1}+\ldots+G_{x_N}=0$. First we will construct a set of
points such that the norm $\|G_{x_1}+\ldots +G_{x_N}\|_w$ is
small, and then we will use the Brouwer fixed point theorem to
show that there exists a collection of points $\{y_1,\ldots,y_N\}$
``close'' to $\{x_1,\ldots,x_N\}$ with $\|G_{y_1}+\ldots
+G_{y_N}\|_w=0$.

 Let
$\mathcal{R}=\left\{R_1,\ldots, R_N\right\}$ be a finite
collection of closed, non-overlapping (i.e., having no common
interior points) regions $R_i\subset S^d$ such that $\cup_{i=1}^N
R_i=S^d$. The partition $\mathcal{R}$ is called area-regular if
$\mathrm{vol}R_i:=\int_{R_i}d\mu_d(x)=1/N$, for all
$i=1,\ldots,N$. The partition norm for $\mathcal{R}$ is defined by
$$ \|\mathcal{R}\|:=\max_{R\in\mathcal{R}}\mathrm{diam}\, R. $$
Now we will prove
\begin{lemma}\label{l1} For each $N\in\mathbb{N}$ there exists an area-regular partition $\mathcal{R}=\left\{R_1,\ldots,
R_N\right\}$  of $S^d$ and a collection of points $x_i\in R_i$,
$i=1,\ldots,N$ such that
$$\left\|\frac{G_{x_1}+\ldots+G_{x_N}}{N}\right\|_w\leq\frac {b_d
n^{d/2}}{N^{1/2+1/d}}. $$
\end{lemma}
\begin{proof}
As shown in ~\cite{SK2}, for each $N\in\mathbb{N}$ there exists an
area-regular partition $\mathcal{R}=\left\{R_1,\ldots,
R_N\right\}$ such that $\|\mathcal{R}\|\leq c_{2d}N^{1/d}$ for
some constant $c_{2d}$. For this partition $\mathcal{R}$ we will
estimate the average value of
$\left\|\frac{G_{x_1}+\ldots+G_{x_N}}{N}\right\|_w^2 $, when the
points $x_i$ are uniformly distributed over $R_i$. We have
$$\begin{array}{l} \displaystyle
\frac{1}{\mathrm{vol}R_1\cdots\mathrm{vol}R_N}\int_{R_1\times\cdots\times
R_N}
\left\|\frac{G_{x_1}+\ldots+G_{x_N}}{N}\right\|_w^2d\mu_d(x_1)\cdots
d\mu_d(x_N)= \\
\displaystyle\qquad{}=\frac{1}{\mathrm{vol}R_1\cdots\mathrm{vol}R_N}\int_{R_1\times\cdots\times
R_N}\frac{1}{N^2}\sum_{i,j=1}^N\langle G_{x_i},G_{x_j}\rangle_w
d\mu_d(x_1)\cdots d\mu_d(x_N)\\ \displaystyle \qquad{}=\sum_{i\neq
j}\int_{R_i\times R_j}\langle G_{x_i},G_{x_j}\rangle_w
d\mu_d(x_i)d\mu_d(x_j) +\sum_{i=1}^N \frac 1N \int_{R_i}\langle
G_{x_i}, G_{x_i}\rangle_w d\mu_d(x_i) \\  \displaystyle \qquad{}
=\int_{S^d\times S^d}\langle G_x, G_y\rangle_w d\mu_d(x)d\mu_d(y)+
\\  \displaystyle \qquad{}\qquad{} +\sum_{i=1}^N\left(\frac 1N
\int_{R_i}\langle G_x,G_x\rangle_w d\mu_d(x) -\int_{R_i\times
R_i}\langle G_x,G_y\rangle_w d\mu_d(x)d\mu_d(y)\right)\\
\displaystyle \qquad{} =\int_{S^d\times
S^d}g_w((x,y))d\mu_d(x)d\mu_d(y)+ \\  \displaystyle
\qquad{}\qquad{} +\sum_{i=1}^N\int_{R_i\times
R_i}g_w(1)-g_w((x,y))d\mu_d(x)d\mu_d(y). \end{array}$$

The first term of the sum is equal to zero because for each fixed
$x\in S^d$, the polynomial $g_w((x,y))\in\mathcal{P}_n$. We can
estimate the second term by $$ \sum_{i=1}^N\int_{R_i\times
R_i}g_w(1)-g_w((x,y))d\mu_d(x)d\mu_d(y)\le
\frac1N\max_{R_i\in\mathcal{R}}\max_{x,y\in
R_i}|g_w(1)-g_w((x,y))| $$ $$
\le\frac1N\max_{R_i\in\mathcal{R}}\max_{x,y\in
R_i}g'_w(1)\|x-y\|^2\le \frac1Nc_{1d}n^d\|\mathcal{R}\|^2\le
c_{1d}\frac{c_{2d}^2n^d}{N^{1+2/d}}, $$ where in the last line we
use~\eqref{g3} and~\eqref{g4}. This immediately implies the
statement of the Lemma.
\end{proof}
For a polynomial $Q\in\p_n$ define the norm of the Hessian matrix
on the sphere, as defined by ~\eqref{Hes}, at the point $x_0\in
S^d$ by $$ \left\|\nabla^2 Q(x_0)\right\|=\max_{\|y\|=1}| \nabla^2
Q(x_0)\cdot y\cdot y|,$$ where the maximum is taken over vectors
$y$ orthogonal to $x_0$. We will prove the following estimate
\begin{lemma}\label{norm}
For a polynomial $Q\in\p_n$ and point $x_0\in S^d$
 $$ \left\|\nabla^2 Q(x_0)\right\|\leq (3g''_w(1)+g'_w(1))^{1/2}\|Q\|_w.$$
\end{lemma}
\begin{proof}Fix a unit vector $y_0$ orthogonal to $x_0$ and
define a curve $x(t)$ on the sphere $S^d$ by $$x(t)=x_0 \cos(t)
+y_0\sin(t).$$ For each $t\in \R$ we consider the polynomial
$G_{x(t)}(y)=g_w((x(t),y))\in \p_n$, which
 has the property $\langle Q,G_{x(t)}\rangle_w= Q(x(t)) $ for all $Q\in\p_n$.
 Setting
 $G''=\frac{d^2}{d t^2}G{x(t)}|_{t=0}$, we have that
 \begin{equation}\label{QP''}
  \nabla^2 Q(x_0) \cdot y_0 \cdot
y_0=\frac{d^2}{d t^2}Q(x(t))|_{t=0}=\langle Q,G''\rangle_w.
\end{equation} Hence
$$\left\|\nabla^2 Q(x_0)\right\|\leq \|G''\|_w\|Q\|_w.$$ It
remains to show that $\|G''\|_w=(3g''_w(1)+g'_w(1))^{1/2}$. Since
$$\frac{d^2}{d t^2} G_{x(t)}(y)= \frac{d^2}{d t^2}
g_w((x(t),y)),$$ we obtain
\begin{equation}\label{P''}G''(y)=(y_0,y)^2g_w''((x_0,y))-(x_0,y)g_w'((x_0,y))
.\end{equation} From (\ref{QP''}) and (\ref{P''}) we get by direct
calculation $$\langle G'',G''\rangle_w=\frac{d^2}{d
t^2}G''(x(t))|_{t=0}=3g_w''(1)+g'_w(1). $$ Lemma 2 is proved.
\end{proof}
\noindent Denote by $B^q$ the closed ball of radius $1$ with
center at $0$ in $\R^q$. To prove the following Lemma~3 we use the
Brouwer fixed point theorem \cite{N}\\ {\bf Theorem B.} Let A be a
closed bounded convex subset of $\mathbb{R}^q$ and $H:A\rightarrow
A$ be a continuous mapping on $A$. Then there exists  some $z\in
A$ such that $H(z)=z$.
\begin{lemma}\label{ball}
Let $F:B^q\to\R^q$ be a continuous map such that
$$F(x)=A(x)+G(x), $$ where $A(x)$ is a linear map and for each
$x\in B^q$
\begin{equation}
\label{ii}
\|A(x)\|\geq \alpha \|x\|
\end{equation}
and
\begin{equation}
\label{ii1}
\|G(x)\|\le\alpha \|x\|/2,
\end{equation}
for some $\alpha>0$. Then, the image of $F$ contains the closed ball of
radius $\alpha/2$ with center at $0$.
\begin{proof}
Take an arbitrary $y$, with $\|y\|\le\alpha/2$. It is sufficient
to show that there exists $x\in B^q$ such that $F(x)=y$. The
inequality~\eqref{ii} implies that $\|A^{-1}(y)\|\le 1/2$. Denote
by $K$ the ball of radius $1/2$ with center $0$. Consider a map
$$H_y(z)=-A^{-1}(G(A^{-1}(y)+z)).$$ By~\eqref{ii} and ~\eqref{ii1}
we obtain that $H_y(K)\subset K$. Hence, by the Brouwer fixed
point theorem, there exists $z\in K$ such that $H_y(z)=z$. This
then implies that $$F(A^{-1}(y)+z)=y.$$
\end{proof}\end{lemma}

To prove the principal Lemma~\ref{mainl} we also need a result
which is an easy corollary of Theorem 3.1 in~\cite{MNW}\\ {\bf
Theorem MNW.} There exist constants $r_d$ and $N_d$ such that for
each area-regular partition $\mathcal{R}=\{R_1,\ldots,R_N\}$ with
$\|\mathcal{R}\|<\frac{r_d}m$, each collection of points $x_i\in
R_i$, $i=1,\ldots,N$ and each algebraic polynomial $P$ of total
degree $m>N_d$ the following inequality
\begin{equation}
\label{Mhaskar}
\frac12\int_{S^d}|P(x)|d\mu_d(x)<\frac1N\sum_{i=1}^N|P(x_i)|<\frac32\int_{S^d}|P(x)|d\mu_d(x)
\end{equation}
holds.

Consider the map
 $\Phi:(S^d)^N\to \p_n $ defined by $$\xymatrix{
(x_1,\ldots,x_N)\ar[r]^>(0.8){\Phi} &
\frac{G_{x_1}+\ldots+G_{x_N}}{N}} .$$
\begin{lemma}\label{mainl} Let $x_1,\ldots,x_N \in S^d$ be the collection of points and
 $\mathcal{R}=\{R_1,\ldots,R_N\}$ an area-regular partition such
 that $x_i\in R_i$
  and $\|\mathcal{R}\|\leq \frac {r_d} {2n}$. Then the image of the map $\Phi$ contains a ball of radius
$\rho\geq A_dn^{(-d-2)/2}$ with  center at the point
$G=\frac{G_{x_1}+\ldots+G_{x_N}}{N}$, where $A_d$ is a
sufficiently small constant, depending only on $d$.
\end{lemma}

\begin{proof}
 For each polynomial $P\in\p_n$
 consider
the circles  on $S^d$ given by $$\tilde{x}_i(t)=x_i \cos( \|\nabla
P(x_i)\| t)+y_i \sin( \|\nabla P(x_i)\| t),$$ where
$y_i=\frac{\nabla P(x_i)}{\|\nabla P(x_i)\|}$, $
i=\overline{1,\ldots, N} $.   Define the map $X:\p_n\to(S^d)^N$ by
$$X(P)=(x_1(P),\ldots,
x_N(P)):=(\tilde{x}_1(1),\ldots,\tilde{x}_N(1)).$$
 Now we will consider the composition  $L=\Phi\circ X:\p_n\to\p_n$
 which takes the form
 $$L(P)=\frac{G_{x_1(P)}+\ldots+G_{x_N(P)}}{N}.$$
 For each $Q\in \p_n$ one can take the Taylor expansion
 \begin{equation}\label{taylor} \langle G_{\tilde{x}_i(t)},Q\rangle_w =Q(\tilde{x}_i(t))=Q(x_i)+\frac{d}{d
t}Q(\tilde{x}_i(0))t+
 \frac 12\cdot \frac{d^2}{d t^2}Q(\tilde{x}_i(t_i))t^2 , \;\;t_i\in[0,t].
 \end{equation}
 Hence, we can represent the function $L(P)$ in the form $$L(P)=L(0)+L'(P)+L''(P). $$
 Here $L'(P)$ is the unique polynomial in $\p_n$ satisfying
 $$\langle L'(P),Q\rangle_w=\frac{1}{N}\sum_{i=1}^N
 (\nabla Q(x_i),\nabla P(x_i))\;\;\mbox{for all}\;\; Q\in \p_n,$$  and $$ L''(P)=L(P)-L(0)-L'(P). $$
First, for each $P\in\p_n$ we will estimate the norm of $L'(P)$
from below. We have
 $$\|L'(P)\|_w\geq\frac{1}{\|P\|_w}\cdot \langle L'(P),P\rangle_w=\frac{1}{\|P\|_w}\cdot\frac{1}{N}\sum_{i=1}^N
 (\nabla P(x_i),\nabla P(x_i)).$$
 Applying~\eqref{Mhaskar} to the polynomial $(\nabla P,\nabla P)$ of degree $\le 2n$, we get $$\frac{1}{N}\sum_{i=1}^N (\nabla P(x_i),\nabla P(x_i))
 \geq \frac12\int_{S^d}(\nabla P(x),\nabla P(x))d\mu_d(x) . $$
 On the other hand, by~\eqref{lap} we have
 $$\int_{S^d}(\nabla P(x),\nabla P(x))d\mu_d(x)=\langle P,\Delta_wP\rangle_w=\|P\|^2_w .$$
 This gives us the estimate \begin{equation}\label{l'} \|L'(P)\|_w\geq \frac12\|P\|_w. \end{equation}
 Now we will estimate the norm of $L''(P)$ from above. By (\ref{taylor}) we have
 $$\langle L''(P),Q\rangle_w=\frac{1}{2N}\sum_{i=1}^N \frac{d^2}{dt^2}Q(\tilde{x}_i(t_i)),
 $$for some $t_i\in [0,1]$. Since the following equality holds
 $$\frac{d^2}{dt^2}Q(\tilde{x}_i(t))=\nabla^2 Q\cdot\frac{d \tilde{x}_i(t)}
 {d t}\cdot\frac{d \tilde{x}_i(t)}{d t},$$
 Lemma \ref{norm}  implies that
 $$|\frac{d^2}{dt^2}Q(\tilde{x}_i(t))|\leq (3g_w''(1)+g_w'(1))^{1/2}\|\frac{d \tilde{x}_i}{d
 t}\|^2\cdot\|Q\|_w.
 $$ It follows from the identity$$\|\frac{d \tilde{x}_i}{d
 t}(t)\|=\|\nabla P(x_i)\| $$ and estimates \eqref{g4},
\eqref{g5}
 that
 $$|\frac{d^2}{dt^2}Q(\tilde{x}_i(t))|\leq
  c_{3d} n^{(d+2)/2}\|\nabla P(x_i)\|^2\cdot\|Q\|_w.
 $$
This inequality yields immediately
 $$|\langle L''(P),Q \rangle_w|=|\frac {1}{2N}\sum_{i=1}^N \frac{d^2}{dt^2}Q(\tilde{x}_i(t_i))
|\leq \frac{c_{3d} n^{(d+2)/2}\|Q\|_w}
 {N}\sum_{i=1}^N\|\nabla P(x_i)\|^2.$$
 Applying again~\eqref{Mhaskar}, we obtain
 $$\frac{1}{N}\sum_{i=1}^N\|\nabla P(x_i)\|^2\leq
 \frac32\|P\|_w^2.
 $$
 So, for each $Q\in \p_n$ we have that
 $$ |\langle L''(P),Q \rangle_w |\leq \frac32 c_{3d} n^{(d+2)/2}
 \|P\|_w^2\cdot\|Q\|_w.$$ Thus, we get \begin{equation}\label{l''} \|L''(P)\|_w\leq
 \frac32 c_{3d} n^{(d+2)/2}
 \|P\|_w^2.
 \end{equation}
 Lemma \ref{ball} combined with inequalities (\ref{l'}) and (\ref{l''})
 implies that the image of $L$, and hence the image of $\Phi$, contains a ball of radius
 $\rho\geq A_dn^{(-d-2)/2}$ around $L(0)=G$, where $A_d=1/6c_{3d}$, proving the lemma.
\end{proof}
\begin{proof}[Proof of Theorem 1.]
By Lemma 1, there exists an area-regular partition
$\mathcal{R}=\left\{R_1,\ldots, R_N\right\}$ such that
$\|\mathcal{R}\|\leq c_{2d}N^{1/d}$, and a collection of points
$x_i\in R_i$, $i=1,\ldots,N$ such that
$$\left\|\frac{G_{x_1}+\ldots+G_{x_N}}{N}\right\|_w\leq\frac {b_d
n^{d/2}}{N^{1/2+1/d}}.$$ Take $N$ large enough such that $N>N_d$
and $\frac{c_{2d}}{N^{1/d}}<\frac{r_d}{2n}$, where $N_d$ and $r_d$
are defined by Theorem MNW. Applying Lemma~\ref{mainl} to the
partition $\mathcal{R}$ and the collection of points $x_1,\ldots,
x_N$, we obtain immediately that $G_{y_1}+\ldots+G_{y_N}=0$ for
some $y_1,\ldots,y_N\in S^d$ if $$ \frac {b_d
n^{d/2}}{N^{1/2+1/d}}<A_dn^{(-d-2)/2}. $$ So, we can choose a
constant $c_d$ such that the last inequality holds for all
$N>c_dn^\frac{2d(d+1)}{d+2}$. Theorem 1 is proved.
\end{proof}

\end{document}